\documentclass[12pt,oneside,a4paper]{amsart}
\usepackage[british]{babel}

\usepackage{graphicx}
\usepackage{amscd}
\usepackage{amsmath}
\usepackage{amsfonts}
\usepackage{amssymb}
\usepackage{mathrsfs}
\usepackage{comment}
\excludecomment{mycomment}
\usepackage{color}
\usepackage[usenames,dvipsnames,table,xcdraw]{xcolor}
\usepackage{pdfsync}
\usepackage{bbm}
\usepackage{dsfont}
\usepackage[colorlinks=true]{hyperref}
\usepackage{enumerate}
\usepackage{booktabs}
\usepackage{float}
\usepackage{wrapfig}
\usepackage{caption}
\usepackage{subcaption}
\usepackage{longtable}
\usepackage{listings}
\usepackage[T1]{fontenc}
\usepackage[scaled]{beramono}
\usepackage{tikz}
\usepackage{pgffor}
\usepackage[all]{xy}
\usepackage{bbold}
\usepackage{enumerate}
\definecolor{trp}{rgb}{1,1,1}

\definecolor{red}{rgb}{1,0,.2}

\usepackage{amsthm}
\theoremstyle{plain}
\newtheorem{theorem}{Theorem}[section]

\newtheorem{corollary}[theorem]{Corollary}

\newtheorem{definition}[theorem]{Definition}
\newtheorem{example}[theorem]{Example}

\newtheorem{lemma}[theorem]{Lemma}

\newtheorem{method}{Method}

\newtheorem{remark}[theorem]{Remark}

\newtheorem{assumption}{Assumption}
\numberwithin{equation}{section}

\newcommand{\iiv}{\overline{\imath}}
\newcommand{\jjv}{\overline{\jmath}}

\newcommand*{\arabicdec}[1]{\the\numexpr\value{#1}\relax}

\usepackage{anysize}

\usepackage{caption}

\definecolor{blue}{rgb}{0,0,1}

\definecolor{red}{rgb}{1,0,.7}

\usepackage{tikz}
\usetikzlibrary{cd,decorations.pathreplacing,positioning,arrows}

\usepackage{xr}

\begin{document}
\title[Special Families of CPLIFS]{Special families of piecewise linear iterated function systems}

\author{R. D\'aniel Prokaj}
\address{ R. D\'aniel Prokaj, Alfréd Rényi Institute of Mathematics, Reáltanoda u. 13-15.,
H-1053 Budapest, Hungary} \email{prokajrd@math.bme.hu}

\author{K\'aroly Simon}
\address{K\'aroly Simon, Department of Stochastics, Institute of Mathematics,
Budapest University of Technology and Economics, Műegyetem rkp. 3.,
H-1111 Budapest, Hungary, and  ELKH-BME Stochastics Research Group, P.O. Box 91, 1521 Budapest, Hungary}
\email{simonk@math.bme.hu}

\begin{abstract}
This paper investigates the dimension theory of some families of continuous piecewise linear iterated function systems. For one family, we show that the Hausdorff dimension of the attractor is equal to the exponential growth rate obtained from the most natural covering system. We also prove that for Lebesgue typical parameters, the 1-dimensional Lebesgue measure of the underlying attractor is positive, if this number is bigger than 1, and all the contraction ratios are positive.
\end{abstract}

\maketitle


\thispagestyle{empty}

\section{Introduction}

Let $m>0$ and $\mathcal{F}=\left\{ f_k:\mathbb{R}\to \mathbb{R}\right\}_{k=1}^m$ be a finite list of strict contractions over $\mathbb{R}$. We call $\mathcal{F}$ an iterated function system. It is well known that there exists a unique non-empty set $\Lambda\subset \mathbb{R}$ for which
\begin{equation}\label{sf95}
  \Lambda = \bigcup_{k=1}^m f_k(\Lambda).
\end{equation} 
This set $\Lambda$ is called the attractor of $\mathcal{F}$. 

For each $k\in[m]$, we always assume that $f_k$ is a continuous piecewise linear function which is a strict contraction over any interval $I\subset \mathbb{R}$ with non-zero slopes. That is, $f_k$ is a continuous function that has different slopes over finitely many intervals. We call these systems \texttt{Continuous Piecewise Linear Iterated Function Systems} or just \texttt{CPLIFS} for short. We will often refer to a CPLIFS whose functions are injective as an \texttt{injective CPLIFS}.
 
Write $\rho_k$ for the biggest slope of $f_k$ in abolute value, $\forall k\in[m]$.
Let us call a CPLIFS \texttt{small} if $\sum_{k=1}^m \vert\rho_k\vert<1$ and  
\begin{itemize}
    \item $\forall k\in[m]:\vert\rho_k\vert <\frac{1}{2}$, if $\mathcal{F}$ is injective;
    \item $\forall k\in[m]:\vert\rho_k\vert <\frac{1}{3}$, otherwise.
\end{itemize}
In \cite{prokaj2021piecewise}, we showed that if $\mathcal{F}$ is small, then typically the dimensions of the attractor $\Lambda$ are equal. For the definition of these fractal dimensions, we refer the reader to \cite{falconer1997techniques}.
\begin{theorem}[Theorem 2.1 of \cite{prokaj2021piecewise}]\label{sf94}
  Let $\mathcal{F}$ be a $\dim_{\rm P}$-typical small CPLIFS with attractor $\Lambda$. Then
  \begin{equation}\label{sf93}
    \dim_{\rm H}\Lambda = \dim_{\rm B}\Lambda .
  \end{equation}
\end{theorem} 
The equality of the box and Hausdorff dimensions of the attractor is thus typical in some sense. Namely, we parametrized each piecewise linear function by its slopes, its value at zero, and by the points where it changes slope. Then for a fixed vector of slopes, we call a property $\mathfrak{P}$ \texttt{$\dim_{\rm P}$-typical} if the packing dimension of the set of those parameters for which the appropriate CPLIFS does not satisfy property $\mathfrak{P}$ is strictly smaller than the packing dimension of the whole parameter space (See Definition \ref{cs67}).

If for a CPLIFS $\mathcal{F}=\{f_k\}_{k=1}^m$ none of the functions change slope on the attractor $\Lambda$, than we call $\mathcal{F}$ a \texttt{regular CPLIFS}. We detailed in \cite[Section~4]{prokaj2021piecewise} how regularity implies the existence of a graph-directed IFS with the same attractor $\Lambda$. The proof of Theorem \ref{sf94} strongly relies on the following result.
\begin{theorem}[Theorem 2.3 of \cite{prokaj2021piecewise}]\label{sf91}
  A $\dim_{\rm P}$-typical small injective CPLIFS is regular.
\end{theorem}
Following this line, we are going to show that \eqref{sf93} holds for some non-regular families as well, without restrictions on the slopes of the functions.

\subsection{Plan of the paper}

We will show that in the case of some specific families of injective CPLIFS, regularity and smallness are not necessary to obtain \eqref{sf93}.
In section \ref{sf90}, we work with CPLIFS where the functions can only change slope at some well defined points on the attractor. We will show how to construct an associated graph-directed iterated function system for these CPLIFS. Using this construction, we may apply theorems known for graph-directed systems to conclude the equality of dimensions.

In section \ref{sf89}, with the help of the theory of P. Raith and F. Hofbauer, we prove that a condition we call IOSC (see Definition \ref{sf87}) implies \eqref{sf93} for injective CPLIFS.

By using the most natural covering system of the attractor of an IFS, we define a pressure function and call its unique zero the natural dimension of the system.
We close the discussion by proving that the Lebesgue measure of the attractor of a Lebesgue typical CPLIFS is always positive given that the natural dimension of the system is bigger than $1$.

\section{Preliminaries}

\subsection{Continuous piecewise linear iterated function systems}

\begin{figure}[t]
  \centering
  \includegraphics[width=9cm]{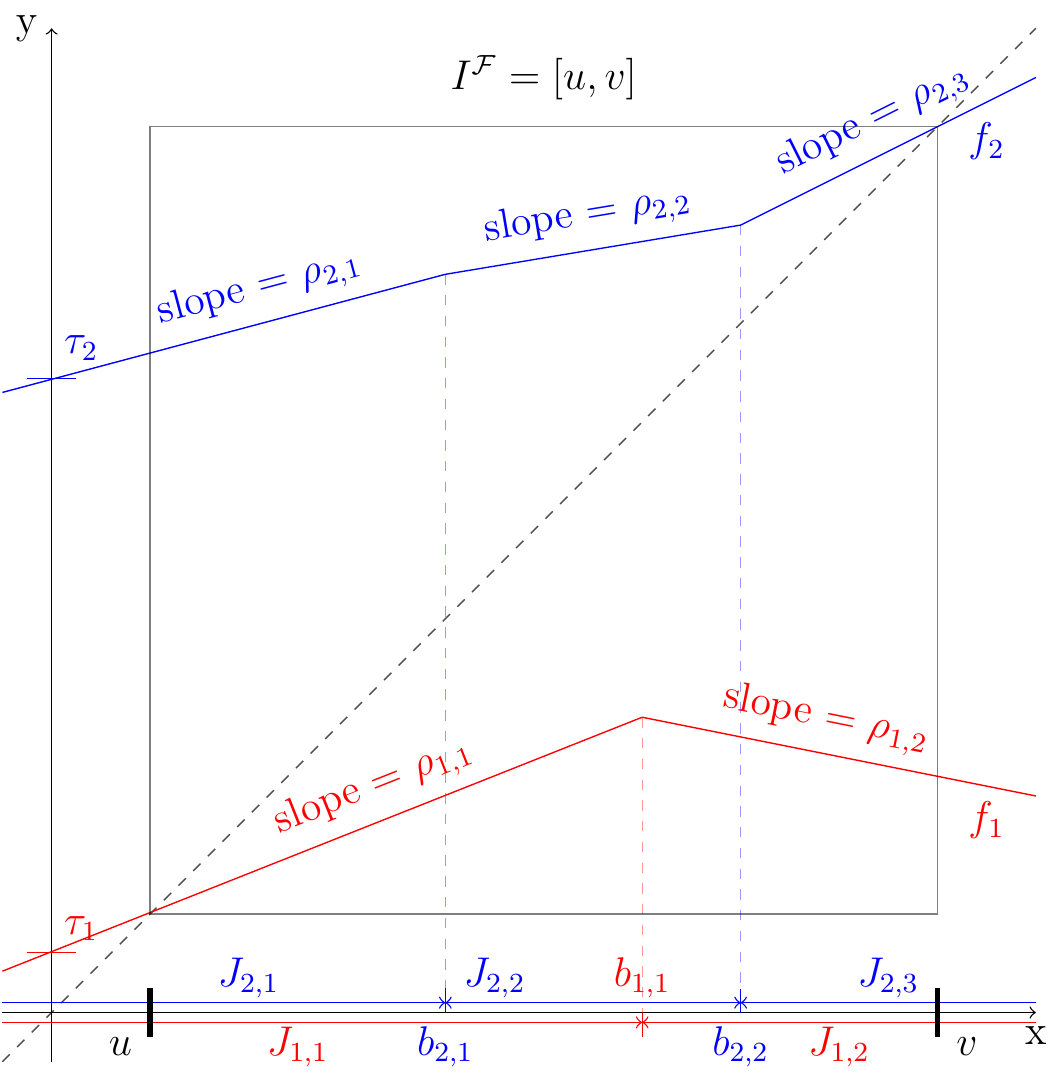}
  \caption{A general CPLIFS with the related notations.}\label{cv47}
\end{figure}

Let $\mathcal{F}=\left\{f_k\right\}_{k=1}^{m }$ be a CPLIFS and $\Sigma:=[m]^{\mathbb{N}}$ be its symbolic space. We denote the set of finite words by $\Sigma^{\ast}$ and the set of $n$ length words by $\Sigma^n$ for any $n\in \mathbb{N}^{+}$.
We write $l(k)$ for the number of breaking points of $f_k$ for $k\in[m]$,
and we say that the \texttt{type of the CPLIFS} is the vector
  \begin{equation}\label{cv44}
  \pmb{\ell }=(l(1), \dots ,l(m)).
\end{equation}
For example, the type of the CPLIFS on Figure \ref{cv47} is
$\pmb{\ell }=(l(1),l(2))=(1,2)$. If $\mathcal{F}$ is a CPLIFS of type $\pmb{\ell }$,  then we write
$$\mathcal{F}\in\mathrm{CPLIFS}_{\pmb{\ell }}.$$
 The breaking points of $f_k$
are denoted by $b_{k,1} < \cdots < b_{k,l(k)}$. We sometimes write $B(k):=\{b_{k,1},\dots ,b_{k,l(k)}\}$ for the set of breaking points of $f_k$.
Let $L:=\sum\limits_{k=1}^{m}l(k)$ be the \texttt{total number of breaking points} of the functions of $\mathcal{F}$ with multiplicity, as some of the breaking points of two different elements of $\mathcal{F} $ may coincide.
We arrange all the breaking points in an
$L$ dimensional vector $\mathfrak{b}\in \mathbb{R}^L$
 in a way described below. First we partition $[L]=\left\{1, \dots ,L\right\}$ into blocks of length $l(k)$ for $k\in[m]$. The $k$-th block is
 \begin{equation}\label{cv52}
 L^k:=\left\{p\in\mathbb{N} :1+\sum\limits_{j=1}^{k-1}l(j) \leq p
 \leq \sum\limits_{j=1}^{k}l(j)
\right\}
 \end{equation}
 where $\sum\limits_{j=1}^{k-1}$ is meant to be $0$ when $k=1$.
 We use this convention without further mentioning it throughout the paper. 
 The breaking points of $f_k$ will make the coordinates of $\mathfrak{b}$ indexed by the block $L(k)$ in increasing order.
\begin{equation}\label{cv58}
  \mathfrak{b}
  =
  (\underbrace{b_{1,1}, \dots ,b_{1,l(1)}}_{L^1},
  \underbrace{b_{2,1}, \dots ,b_{2,l(2)}}_{L^2},
  \dots ,
  \underbrace{b_{m,1}, \dots ,b_{m,l(m)}}_{L^m}).
\end{equation}
The set of breaking points vectors $\mathfrak{b}$ for a type $\pmb{\ell }$ CPLIFS is
\begin{equation}\label{cv59}
  \mathfrak{B}^{\pmb{\ell }}:=
  \left\{(x_1,\dots ,x_m)\in\mathbb{R}^{L}
  : x_i<x_j \mbox{ if } i<j \mbox{ and } \exists k\in [m] \mbox{ with }
  i,j\in L^k
  \right\}.
\end{equation}
The $l(k)$ breaking points of the piecewise linear continuous function $f_k$ determines the $l(k)+1$  \texttt{intervals of linearity} $J_{k,i}^{\mathfrak{b}}$, among which the first and the last are actually half lines:
\begin{equation}\label{cv57}
 J_{k,i}:=  J_{k,i}^{\mathfrak{b}}:=
  \left\{
    \begin{array}{ll}
      (-\infty ,b_{k,i}), & \hbox{if $i=1$;} \\
      (b_{k,i-1},b_{k,i}), & \hbox{if $2\leq i\leq l(k)$;} \\
      (b_{k,l(k)},\infty ), & \hbox{if $i=l(k)+1$.}
    \end{array}
  \right.
\end{equation}
The derivative of $f_k$ exists on $J_{k,i}$ and is equal to the constant
\begin{equation}\label{cs68}
  \rho_{k,i}:\equiv f'_k|_{J_{k,i}}.
\end{equation}
We arrange the contraction ratios $\rho_{k,i}\in (-1,1)\setminus \{ 0\}$ into a vector $\pmb{\rho}$ in an analogous way as we arranged the breaking points into a vector $\mathfrak{b}$ in \eqref{cv58}, taking into account that there is one more contraction ratio for each $f_k$ than breaking points:
\begin{equation}\label{cv56}
  \pmb{\rho}:=\pmb{\rho}_{\mathcal{F}}:=
  (
  \underbrace{\rho_{1,1}, \dots ,\rho_{1,l(1)+1}}_{\widetilde{L}^1},
 \dots
 ,\underbrace{\rho_{m,1}, \dots ,\rho_{m,l(m)+1}}_{\widetilde{L}^m}
  )\in \left( (-1,1)\setminus \{ 0\}\right)^{L+m},
\end{equation}
where
\begin{equation}\label{cv51}
  \widetilde{L}^k:=
  \left\{ p \in\mathbb{N}:
  1+\sum\limits_{j=1}^{k-1}\left(1+l(j)\right) \leq p  \leq
  \sum\limits_{j=1}^{k}\left(1+l(j)\right) \right\}.
\end{equation}
We call $\pmb{\rho}$ the \texttt{vector of contractions}.
The set of all possible values of $\pmb{\rho}$ for an $\mathcal{F}\in \mathrm{CPLIFS}_{\pmb{\ell }}$
is
\begin{equation}\label{cv50}
  \mathfrak{R}^{\pmb{\ell }}:
  =
  \left\{\pmb{\rho}
  \in\left((-1,1)\setminus \{ 0\}\right)^{L+m}:
  \forall k\in[m],\
\!  \forall i,i+1\in\widetilde{L}^k,\!
  \rho_{i}\ne \rho_{i+1}
  \right\},
\end{equation}
where $\pmb{\rho}=(\rho_1, \dots ,\rho_{L+m})$.
Write $\rho_k:=\max_{j\in[l(k)+1]}\vert\rho_{k,j}\vert$ for each $k\in[m]$, and write $\rho_{\max}, \rho_{\min}$ for the biggest and the smallest contraction ratios of $\mathcal{F}$ respectively.
Let us use the usual notations $\rho_{k_1 \dots k_n}:=\rho_{k_1}\cdots\rho_{k_n}$ and $f_{k_1 \dots k_n}:=f_{k_1}\circ\cdots\circ f_{k_n}$.
Clearly,
\begin{equation}\label{ct67}
  \vert f'_{k_1 \dots k_n}(x)\vert \leq \rho_{k_1 \dots k_n},
  \mbox{ for all } x \mbox{, where the derivative exists}.
\end{equation}

Finally, we write
\begin{equation}\label{cv55}
  \tau_k:=f_k(0), \mbox{ and }
  \pmb{\tau}:=(\tau_1, \dots ,\tau_m)\in\mathbb{R}^{m}.
\end{equation}
So, the parameters that uniquely determine an $\mathcal{F}\in\mathrm{CPLIFS}_{\pmb{\ell }}$
can be organized into a vector
\begin{equation}\label{cv53}
  \pmb{\lambda}=(\mathfrak{b},\pmb{\tau},\pmb{\rho})\in
  \pmb{\Gamma}^{\pmb{\ell }}:=\mathfrak{B}^{\pmb{\ell }}\times\mathbb{R}^m\times\mathfrak{R}^{\pmb{\ell }}
   \subset \mathbb{R}^L\times\mathbb{R}^m\times\mathbb{R}^{L+m}=
   \mathbb{R}^{2L+2m}.
\end{equation}
We call $\pmb{\Gamma}^{\pmb{\ell}}$ the \texttt{parameter space} of CPLIFS of type $\pmb{\ell}$.
For a $\pmb{\lambda}\in \pmb{\Gamma}^{\pmb{\ell }}$ we write $\mathcal{F}^{\pmb{\lambda}}$ for the corresponding CPLIFS and $\Lambda^{\pmb{\lambda}}$ for its attractor. Similarly, for an $\mathcal{F}\in\mathrm{CPLIFS}_{\pmb{\ell}}$ we write
$\pmb{\lambda}(\mathcal{F})$ for the corresponding element of $\pmb{\Gamma}^{\pmb{\ell }}$. We will refer to $(\mathfrak{b},\pmb{\tau}):=(\mathfrak{b},\pmb{\tau})(\mathcal{F})$ as the \texttt{translation parameters} of $\mathcal{F}$.

Let $S_{k,i}$ be the contracting similarity on $\mathbb{R}$ that satisfies 
$S_{k,i}|_{J_{k,i}}\equiv f_k|_{J_{k,i}}$. 
We say that $\mathcal{S}_{\mathcal{F}}:=\left\{S_{k,i}\right\}_{k\in[m],i\in [l(k)+1]}$ is the \texttt{self-similar IFS  generated by the CPLIFS  $\mathcal{F}$}. 
Obviously, $S_{k,i}^{\prime}=\rho_{k,i}$.

With the help of these notations we can define properly the $\dim_{\rm P}$-typicality.
\begin{definition}\label{cs67}
  Let $\mathfrak{P}$ be a property that makes sense for every CPLIFS $\mathcal{F}$.   
  For a contraction vector $\pmb{\rho}\in\mathfrak{R}^{\pmb{\ell} }$ we consider the (exceptional) set  
  \begin{equation}\label{cs71}
    E_{\mathfrak{P},\pmb{\ell }}^{\pmb{\rho}}=:
    \left\{
    (\mathfrak{b},\pmb{\tau})\in  \mathfrak{B}^{\pmb{\ell }}\times
    \mathbb{R}^m
    :
    \mathcal{F}^{(\mathfrak{b},\pmb{\tau},\pmb{\rho})} \mbox{ does not have property } \mathfrak{P}
    \right\}.
  \end{equation}
  We say that \texttt{property $\mathfrak{P}$ holds $\dim_{\rm P}$-typically} if for all
  type $\pmb{\ell}$ and for all $\pmb{\rho}\in\mathfrak{R}^{\pmb{\ell} }$ we have
  \begin{equation}\label{cs70}
     \dim_{\rm P} E_{\mathfrak{P},\pmb{\ell}}^{\pmb{\rho}} < L+m,
  \end{equation}
  where $\pmb{\ell }=(l(1), \dots ,l(m))$ and $L=\sum\limits_{k=1}^{m}l(k)$ as above.
\end{definition}

\subsection{Graph-directed iterated function systems}\label{sf92}

We present here the most important notations and results related to self-similar Graph-Directed Iterated function Systems (GDIFS).
In this subsection we follow  the book \cite{falconer1997techniques} and the papers \cite{mauldin1988hausdorff} and \cite{keane2003dimension}.
Just like in the last reference, we don't assume any separation conditions.

To define the graph-directed iterated function systems we need a directed graph $\mathcal{G}=\left(\mathcal{V,E}\right)$.
We label the vertices of this graph with the numbers $\lbrace 1,2,...,q\rbrace$, where $\vert\mathcal{V}\vert =q$. This $\mathcal{G}$ graph is not assumed to be simple, it might have multiple edges between the same vertices, or even loops. For an edge $e=(i,j)\in\mathcal{E}$ we write $s(e):=i$ for the source and $t(e):=j$ for the target of $e$.
Denote with $\mathcal{E}_{i,j}$ the set of directed nodes from vertex $i$ to vertex $j$, and write $\mathcal{E}_{i,j}^k$ for the set of length $k$ directed paths between $i$ and $j$. Similarly, we write $\mathcal{E}^n$ for the set of all paths of length $n$ in the graph.
We assume that $\mathcal{G}$
is strongly connected. That is for every $i,j\in\mathcal{V}$ there is a directed path in $\mathcal{G}$ from $i$ to $j$.
\par For all edge $e\in \mathcal{E}$ given a contracting similarity mapping $F_e:\mathbb{R} \rightarrow \mathbb{R}$. The contraction ratio is denoted by $r_e \in (-1,1)\setminus \{ 0\}$. Let $e_1 \dots e_n$ be a path in $\mathcal{G}$.
Then we write
$F_{e_1 \dots e_n}:=F_{e_1}\circ\cdots\circ F_{e_n}$.
It follows from the proof of \cite[Theorem 1.1]{mauldin1988hausdorff}
that there exists a unique family of non-empty compact sets $\Lambda_1,...,\Lambda_q$ labeled by the elements of $\mathcal{V}$, for which
\begin{equation}\label{cv66}
\Lambda_i=\bigcup\limits_{j=1}^q
\bigcup\limits_{e\in\mathcal{E}_{i,j}}F_e(\Lambda_j),\quad i=1, \dots ,q.
\end{equation}
We call the sets $\lbrace \Lambda_1,...,\Lambda_q\rbrace$ \texttt{graph-directed sets}, and we call $\Lambda:=\bigcup\limits_{i=1}^{q}\Lambda_i$ the attractor of the \texttt{self-similar graph-directed IFS} $\mathcal{F}=\{ F_e\}_{e\in\mathcal{E}}$.
We abbreviate it \texttt{self-similar GDIFS}.

By iterating \eqref{cv66} we obtain
\begin{equation}\label{sf86}
\Lambda_i=\bigcup_{j=1}^q \bigcup_{(e_1,...e_k)\in \mathcal{E}_{i,j}^k} F_{e_1 \dots e_k} (\Lambda_j).
\end{equation}

To get the most natural guess for the dimension of $\Lambda$, we define a $q\times q$ matrix with the following entries
\begin{equation}\label{ct50}
C^{(s)}=(c^{(s)}(i,j))_{i,j=1}^{q} \mbox{ and }
c^{(s)}(i,j)= \left\{ \begin{array}{ll}
0, & \hbox{if $\mathcal{E}_{i,j}= \emptyset $;} \\
\sum\limits_{e\in \mathcal{E}_{i,j}} |r_{e}|^{s}, & \hbox{otherwise;}
\end{array}
\right.
\end{equation}
where $s\geq 0$ is a parameter.
The spectral radius of $C^{(s)}$ is denoted by $\varrho(C^{(s)})$. Mauldin and Williams \cite[Theorem 2]{mauldin1988hausdorff} proved that the function
$s\mapsto \varrho(C^{(s)})$ is strictly decreasing, continuous, greater than $1$ at $s=0$, and less than $1$ if $s$ is large enough.

\begin{definition}\label{cv65}
For the self-similar GDIFS $\mathcal{F}=\left\{f_e\right\}_{e\in \mathcal{E}}$ there exists a unique $\alpha=\alpha(\mathcal{F})$ satisfying
\begin{equation}\label{cv64}
\varrho (C^{(\alpha )})=1.
\end{equation}
\end{definition}

The relation of $\alpha$ to the dimension of the attractor is given by the following theorem. It was published in \cite{mauldin1988hausdorff}
apart from the box dimension part, which is from \cite{falconer1997techniques}.
\begin{theorem}\label{cs14}
Let $\mathcal{F}=\left\{F_e\right\}_{e\in E}$ be a self-similar GDIFS as above.
In particular, the graph $\mathcal{G}=(\mathcal{V},E)$ is strongly connected and let $\Lambda$ be the attractor.
  \begin{enumerate}[{\bf (a)}]
    \item $\dim_{\rm H} \Lambda \leq \alpha$.
    \item Let $I_k$ be the interval spanned by $\Lambda_k$ for all $k\in\mathcal{V}$. If the intervals $\left\{I_k\right\}_{k=1}^{q}$ are pairwise disjoint then
  $\dim_{\rm H} \Lambda=\dim_{\rm B}\Lambda=\alpha $. Moreover,
  $0<\mathcal{H}^ {\alpha}(\Lambda)<\infty $.
  \end{enumerate}
\end{theorem}

The equality of the dimensions of the attractor is also known for typical translation parameters. Let $\mathcal{F}^{\mathbf{t}}=\{f_e(x)=\lambda_e\cdot x+ t_e\}_{e\in \mathcal{E}}$ be a family of self-similar GDIFS paramterized by the vector of translations $\mathbf{t}=(t_e)_{e\in \mathcal{E}}$. 

\begin{theorem}[Theorem 1 of \cite{keane2003dimension}]\label{sf72}
Let $\mathcal{F}^{\mathbf{t}}$ be a self-similar GDIFS with directed graph $\mathcal{G}^{\mathbf{t}}$ and attractor $\Lambda^{\mathbf{t}}$. Suppose that  $\mathcal{G}^{\mathbf{t}}$ is strongly connected and all the functions in $\mathcal{F}^{\mathbf{t}}$ has positive slopes.

Let $N:=|E|$ and write $\mathcal{L}_N$ for the $N$ dimensional Lebesgue measure. Then, for $\mathcal{L}_N$-almost every $\mathbf{t}\in \mathbb{R}^N$ we have 
  \begin{enumerate}[{\bf (a)}]
    \item $\dim_{\rm H} \Lambda^{\mathbf{t}} = \min\{1,\alpha\}$,
    \item if $\alpha>1$, then $\mathcal{L}_1(\Lambda^{\mathbf{t}})>0$.
  \end{enumerate}
\end{theorem}

To state another useful theorem on GDIFS, we need to define a separation condition, that is mostly used for one dimensional self- similar iterated function systems. 
Hochman \cite{hochman2014self} introduced the notion of exponential separation for self-similar IFSs. To state it, first we need to define the distance of two similarity mappings
$g_1(x)=\varrho_1x+\tau_1$ and $g_2(x)=\varrho_2x+\tau_2$, $\varrho_1,\varrho_2\in (-1,1)\setminus \left\{0\right\}$,
 on
$\mathbb{R}$. Namely,
\begin{equation}\label{cv92}
  \mathrm{dist}\left(g_1,g_2\right):=
  \left\{
    \begin{array}{ll}
      |\tau_1-\tau_2|, & \hbox{if $\varrho_1=\varrho_2$;} \\
      \infty , & \hbox{otherwise.}
    \end{array}
  \right.
\end{equation}
\begin{definition}\label{cv91}
Given a self-similar IFS $\mathcal{F}=\left\{f_k(x)\right\}_{k=1}^{m}$ on $\mathbb{R}$.
We say that $\mathcal{F}$ satisfies the \texttt{Exponential Separation Condition (ESC)} if there exists a $c>0 $ and a strictly increasing sequence of natural numbers $\left\{n_\ell \right\}_{\ell =1}^{\infty }$ such that
\begin{equation}\label{cv90}
\mathrm{dist}\left(f_{\iiv},f_{\jjv}\right)  \geq c^{n_{\ell }} \mbox{ for all }
\ell  \mbox{ and for all } \iiv,\jjv\in \left\{1, \dots ,M\right\}^{n_\ell },\ \iiv\ne\jjv.
\end{equation}
\end{definition}
We note that the exponential separation condition always holds when an IFS is parametrized by algebraic parameters \cite{hochman2014self}.

\begin{definition}\label{cr35}
  Let $\mathcal{F}=\{F_e\}_{e\in \mathcal{E}}$ be a self-similar GDIFS, with edge set $\mathcal{E}$.
  We call $\mathcal{S}=\left\{S_k(x)=r_{e_k}x+t_{e_k}\right\}_{k=1}^{M}$ \texttt{the self-similar IFS associated with } $\mathcal{F}$. Clearly,
  \begin{equation}\label{cr36}
    S_{k}|_{\Lambda_{t(e_k)}}\equiv F_{e_k}|_{\Lambda_{t(e_k)}}.
  \end{equation}
\end{definition}

We proved in \cite{prokaj2021piecewise}, that the dimensions of the attractor of a self-similar GDIFS are all equal if its generated self-similar IFS satisfies the ESC.   
\begin{theorem}[Corollary 7.2 of \cite{prokaj2021piecewise}]\label{sf77}
  Let $\mathcal{G}=(\mathcal{V},\mathcal{E})$ be a directed graph, and let
  $\mathcal{F}=\left\{F_e\right\}_{e\in \mathcal{E}}$ be a self-similar GDIFS on $\mathbb{R}$ with attractor $\Lambda$. 
  Assume that $\mathcal{G}$ is strongly connected, and that the self-similar IFS $\mathcal{S}$ associated to $\mathcal{F}$ satisfies the ESC. Then
  \begin{equation}\label{cs15}
    \dim_{\rm H}\Lambda = \dim_{\rm B}\Lambda = \min\{1,\alpha\}.
  \end{equation}
\end{theorem}

According to \cite[Theorem~1.10]{Hochman_2015}, the ESC is a $\dim_{\rm P}$-typical property. Hence, Theorem \ref{sf77} extends part (a) of Theorem \ref{sf72} to a wider set of translation and contraction parameters, as the positivity of the slopes is not required anymore.

\begin{remark}\label{sf71}
  Let $\mathcal{F}^{\mathbf{t}}$ be a self-similar GDIFS with directed graph $\mathcal{G}^{\mathbf{t}}$ and attractor $\Lambda^{\mathbf{t}}$. Suppose that  $\mathcal{G}^{\mathbf{t}}$ is strongly connected. Then for a $\dim_{\rm P}$-typical translation vector $\mathbf{t}$ 
  \[
    \dim_{\rm H}\Lambda^{\mathbf{t}} = \dim_{\rm B}\Lambda^{\mathbf{t}} = \min\{1,\alpha\}.
  \]
\end{remark}

In \cite{prokaj2021piecewise}, we showed how to associate a self-similar GDIFS to a regular CPLIFS. This association made it possible to prove that the dimensions of the attractor are equal. 
In the next section, we will associate a self-similar graph-directed iterated function system to some non-regular CPLIFS. This way the theorems showcased in this section can be used to prove results on the attractor.

\section{Breaking points with periodic coding}\label{sf90}

Throuhout this section we will work only with CPLIFS $\mathcal{F}$ that satisfies the following assumption.
\begin{assumption}\label{sf85}
  If a breaking point $b$ of a function in $\mathcal{F}$ falls onto the attractor $\Lambda_{\mathcal{F}}$, then it only has periodic codings in the symbolic space. Precisely,
  \begin{equation}\label{sf84}
    \exists \mathbf{i}\in \Sigma : \Pi(\mathbf{i})=b \implies 
    \mathbf{i} \mbox{ is periodic},
  \end{equation}
  where $\Pi :\Sigma\to\Lambda$ denotes the natural projection defined by $\mathcal{F}$.
\end{assumption}

Let $\mathcal{F}=\{f_k\}_{k=1}^m$ be a CPLIFS that satisfies Assumption \ref{sf85}. Let $b_1, \dots b_Q$ be those breaking points of $\mathcal{F}$ that fall onto the attractor $\Lambda:=\Lambda_{\mathcal{F}}$. According to \eqref{sf84}, they can only have periodic codes: $\mathbf{i}_1, \dots, \mathbf{i}_{Q^{'}}$. Note that, as we have no separation condition on $\mathcal{F}$, some breaking points might have multiple codes, hence $Q \leq Q^{'}$. Further, we write $p_j<\infty$ for the period of $\mathbf{i}_j, \forall j\in[Q^{'}]$ and $P$ for the smallest common multiplier of the numbers $p_1,\dots p_{Q^{'}}$. 

Now we have all the necessary notations to associate a self-similar GDIFS to $\mathcal{F}$. 
Consider the cylinders of level $P$: 
\begin{equation}\label{sf83}
  \Lambda^P := \left\{ f_{\mathbf{j}}(\Lambda): \forall \mathbf{j}\in\Sigma^P \right\}.
\end{equation}
For a $\mathbf{j}\in\Sigma^P$, we call $\phi_A$ the fixed point of the set $A\in \Lambda^P$ if $A=f_{\mathbf{j}}(\Lambda)$ and $f_{\mathbf{j}}(\phi_A)=\phi_A$.

We construct the graph-directed sets from the elements of $\Lambda_P$ in the following way:
\begin{enumerate}
  \item If $A\in\Lambda^P$ does not contain any of the the points $b_1,\dots, b_Q$, then $A$ is a graph-directed set;
  \item If $A\in\Lambda^P$ contains a breaking point as an inner point, then we cut $A$ into two new closed sets $A^{-},A^{+}$ by its fixed point $\phi_A$. The sets $A^{-},A^{+}$ are graph-directed sets. 
\end{enumerate}
That is, we can define the set $G$ of graph-directed sets in the following way. 
\begin{align}\label{sf82}
  \forall A\in \Lambda^P: 
  \forall j\in[Q]: & \; b_j\not\in A \implies A\in G \nonumber\\
  \exists j\in[Q]: & \; b_j\in A \implies A^{-},A^{+}\in G. \nonumber
\end{align}
We say that $\mathbf{i}_A\in\Sigma^P$ is the code of the set $A\in\Lambda^P$ if $\Pi(\mathbf{i}_A)=A$. This way we can define the code of graph directed sets as well. Note that the graph directed sets $A_1,A_2\in G$ will share the same code if $A_1=A_{-}$ and $A_2=A_{+}$ for some $A\in\Lambda^P$.

\begin{lemma}\label{sf81}
  The elements of $G$ do not contain any breaking point as an inner point. 
\end{lemma}

\begin{proof}
  We assumed that all of the codes $\mathbf{i}_1, \dots, \mathbf{i}_{Q^{'}}$ are periodic, and $P$ was defined as the smallest common multiplier of their periods. As $\mathcal{F}$ satisfies Assumption \ref{sf85}, if $j\in[Q], b_j$ is contained in $A\in\Lambda^P$, then $b_j=\phi_A$.
  
  Since we cut the corresponding elements of $\Lambda_P$ into two by their fixed points, it follows that the elements of $G$ can only contain a breaking point $b_j$ as a boundary point.
\end{proof}

To associate a GDIFS to $\mathcal{F}$, we need a directed graph $\mathcal{G}=\{\mathcal{V},\mathcal{E}\}$. We already defined the graph-directed sets as the elements of $G$. 
Accordingly, we define the set of nodes as $\mathcal{V}=\{1,\dots ,\vert G\vert\}$.
For an arbitrary graph directed set $A\in G$, let $\mathbf{i}_{A}\in\Sigma^P$ be its code and $q_A\in \mathcal{V}$ be the node in the graph representing this set. 
The set of edges $\mathcal{E}$ is defined the following way
\begin{equation}\label{sf70}
    A, A^{'}\in G \mbox{, and } f_{\mathbf{i}_A}(A^{'})\in A \implies (A,A^{'})\in \mathcal{E}.
\end{equation}
For an edge $e=(A,A^{'})\in \mathcal{E}$, we define the corresponding contraction as $f_e=f_{\mathbf{i}_A}$. We call the graph-directed system defined by $\mathcal{G}=(\mathcal{V},\mathcal{E})$ and $\{f_e\}_{e\in \mathcal{E}}$ the \texttt{associated graph-directed system of $\mathcal{F}$}, and we denote it by $\mathcal{F}_{\mathcal{G}}$. According to Lemma \ref{sf81}, $\mathcal{F}_{\mathcal{G}}$ is always self-similar. We just obtained the following result.

\begin{theorem}\label{sf69}
   Let $\mathcal{F}$ be a CPLIFS with attractor $\Lambda$. If $\mathcal{F}$ satisfies Assumption \ref{sf85}, then $\Lambda$ is the attractor of a self-similar graph directed iterated function system.  
\end{theorem}

\subsection{Fixed points as breaking points}\label{cp70}

In general, we cannot give a formula for the Hausdorff dimension of the attractor of a CPLIFS, but we can in some special cases, using the previously described construction.
Here we demonstrate it on the case of CPLIFS $\mathcal{F}$ with the following properties:
\begin{enumerate}
    \item $\mathcal{F}$ is injective,
    \item the functions of $\mathcal{F}$ have positive slopes,
    \item the functions of $\mathcal{F}$ can only break at their fixed points,
    \item the first cylinder intervals are disjoint.
\end{enumerate}

We construct the associated directed graph, and then we give a recursive formula for the Hausdorff dimension of these systems.

Let $\mathcal{F}=\{f_i\}_{i=1}^m$ be an injective CPLIFS, and let $\Lambda$ be its attractor. Without loss of generality, we assume that the functions of $\mathcal{F}$ are strictly increasing. For each $i\in[m]$, we denote the fixed point of $f_i$ with $\phi_i$. We further assume that the only breaking point of $f_i$ is $\phi_i$, for every $i\in[m]$.
We call $f_1$ and $f_m$ the maps with the smallest and largest fixed points respectively. Then the interval $I$ that supports the attractor is defined by $\phi_1$ and $\phi_m$. From now on without loss of generality we suppose that $\phi_1=0$ and $\phi_m=1$.

\begin{figure}[b]
    \centering
    \includegraphics[width=13cm]{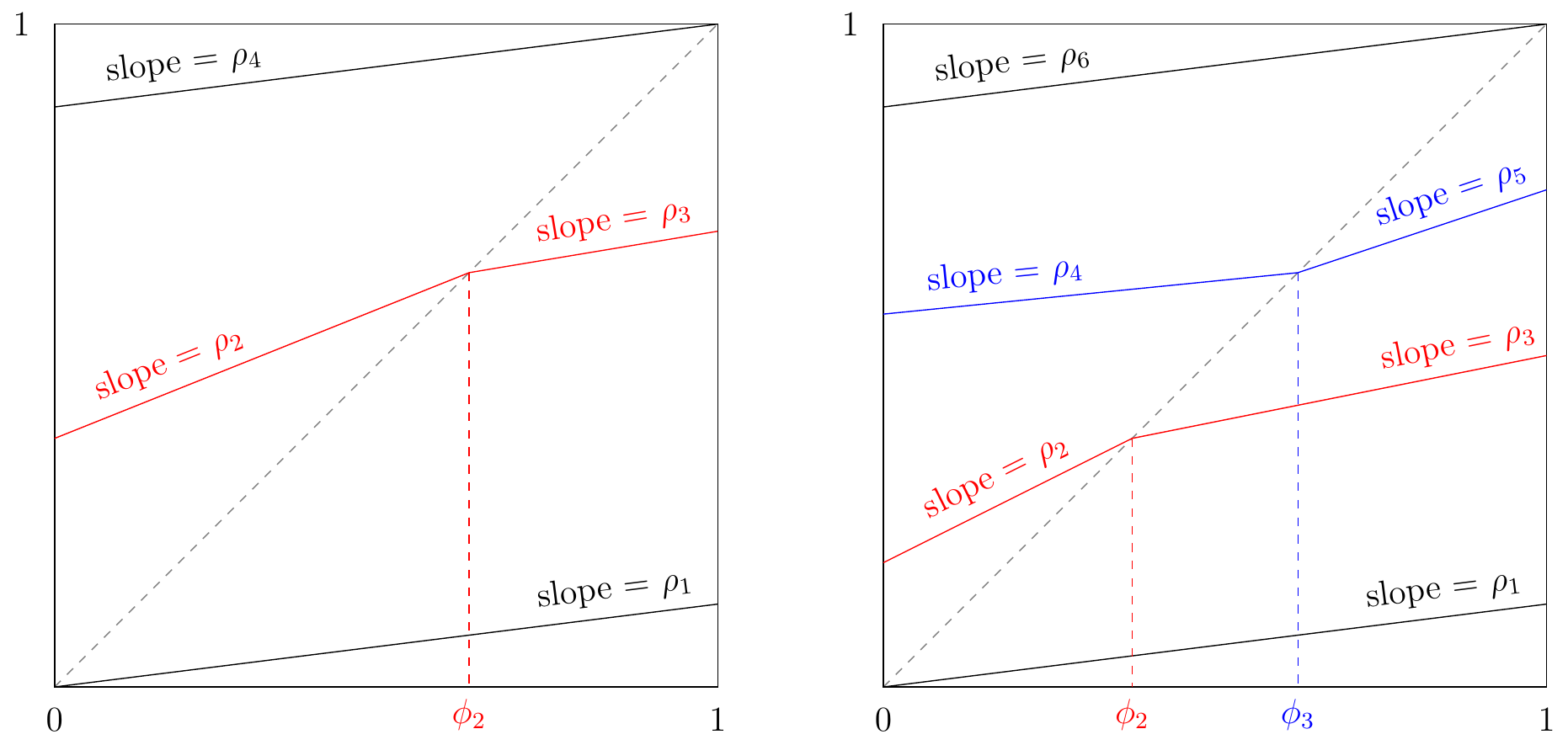}
    \caption{Illustration of CPLIFSs that we discuss in this section with $m=3$ and $4$ functions.}\label{sf56}
  \end{figure}

Let $m=3$, and define the functions of $\mathcal{F}$ as follows.

\begin{equation*}
\begin{split}
f_1(x) &=\rho_1 x, \\
f_3(x) &=\rho_4 x+ (1-\rho_4)
\end{split}
\quad
\begin{split}
f_2(x)= \begin{cases}
		\rho_2 x+\phi_2(1-\rho_2) \text{, if } x\in \left[ 0,\phi_2 \right] \\
		\rho_3 x+\phi_2(1-\rho_3) \text{, if } x\in \left[ \phi_2,1 \right]
		\end{cases}
\end{split}
\end{equation*}

We require here that $\rho_2 \neq \rho_3$ and that $\mathcal{F}$ satisfies the first cylinder intervals of the system are disjoint.
Clearly, $f_2$ only breaks at its fixed point $\phi_2$. Thanks to this property, all elements of the generated self-similar IFS are self-mappings of certain intervals. 
Namely, $S_{2,1}$ is a self-map of $[0,\phi_2]$ and $S_{2,2}$ is a self-map of $[\phi_2,1]$. It implies that we can associate a graph-directed function system with directed graph $\mathcal{G}=\{\mathcal{V},\mathcal{E}\}$, where $\mathcal{V}=[4]$ and the edges are defined by the incidence matrix
\[ \begin{bmatrix}
  1 &1 &1 &1 \\
  1 &1 &0 &0 \\
  0 &0 &1 &1 \\
  1 &1 &1 &1
  \end{bmatrix} .\]
As we did in \ref{ct50}, we define the following matrix
\[ \begin{bmatrix}
\rho_1^s &\rho_1^s &\rho_1^s &\rho_1^s \\
\rho_2^s &\rho_2^s &0 &0 \\
0 &0 &\rho_3^s &\rho_3^s \\
\rho_4^s &\rho_4^s &\rho_4^s &\rho_4^s
\end{bmatrix} .\]

With the help of Theorem \ref{cs14} and the Perron Frobenius Theorem, the Hausdorff dimension of $\Lambda$ equals to the solution of the following equation

\begin{align*}
0 &= \det C^{(s)}_{4}:=\det \begin{bmatrix}
\rho_1^s-1 &\rho_1^s &\rho_1^s &\rho_1^s \\
\rho_2^s &\rho_2^s-1 &0 &0 \\
0 &0 &\rho_3^s-1 &\rho_3^s \\
\rho_4^s &\rho_4^s &\rho_4^s &\rho_4^s-1
\end{bmatrix} \\
&=1 -\rho_1^s -\rho_2^s -\rho_3^s -\rho_4^s +\rho_1^s \rho_3^s +\rho_2^s \rho_3^s +\rho_2^s \rho_4^s.
\end{align*}

Thus the Hausdorff dimension of $\mathcal{F}$ is the unique number $s$ that satisfies the equation
\begin{equation}\label{cp78}
  Q_4(s):=1 -\rho_1^s -\rho_2^s -\rho_3^s -\rho_4^s +\rho_1^s \rho_3^s +\rho_2^s \rho_3^s +\rho_2^s \rho_4^s =0.
\end{equation}
We call $Q_4(s)$ the \texttt{determinant function} of $\mathcal{F}$.  
It is easy to check that \eqref{cp78} gives back the similarity dimension in the self-similar case $(\rho_2=\rho_3)$, thus it is a consistent extension of the dimension theory of self-similar systems.

In a similar fashion we write $Q_{2m-2}(s)$ for the determinant function of a CPLIFS with $m\geq 3$ functions and $C^{(s)}_{2m-2}$ for the corresponding matrix (the matrix of the associated GDIFS minus the appropriate dimensional identity matrix), to keep track of the number of different slopes as parameters in the notation. For instance, if $m=4$, then 
\begin{equation*}
    C^{(s)}_{2m-2} = C^{(s)}_6 = \begin{bmatrix}
        \rho_1^s-1 &\rho_1^s &\rho_1^s &\rho_1^s &\rho_1^s &\rho_1^s \\
        \rho_2^s &\rho_2^s-1 &0 &0 &0 &0 \\
        0 &0 &\rho_3^s-1 &\rho_3^s &\rho_3^s &\rho_3^s \\
        \rho_4^s &\rho_4^s &\rho_4^s &\rho_4^s-1 &0 &0 \\
        0 &0 &0 &0 &\rho_5^s-1 &\rho_5^s \\
        \rho_6^s &\rho_6^s &\rho_6^s &\rho_6^s &\rho_6^s &\rho_6^s-1
        \end{bmatrix}.
\end{equation*}

Fixing the slopes $\rho_1,\dots ,\rho_{2m-2}$ let us express $Q_{2m-2}(s)$ recursively, since this way the determinant of the upper left $2n\times 2n$ block in $C^{(s)}_{2m-2}$ equals to $Q_{2n}(s)$ for each $0<n<m-1$. After expanding the determinant of $C^{(s)}_{2m-2}$ by the second row from below we obtain the following formula 
\begin{equation}\label{cp74}
  Q_{2m-2}(s)=(1-\rho^s_{2m-2}-\rho^s_{2m-3})Q_{2m-4}(s) +\rho^s_{2m-2}\sum_{i=1}^{2m-4} (-1)^i Q_{2m-4,i}(s),
\end{equation}
where $Q_{2m-4,i}(s)$ is the determinant of the matrix obtained by erasing the $i$-th column of $C^{(s)}_{2m-4}$ and then adding the first $(2m-4)$ elements of the last column of $C^{(s)}_{2m-2}$, as a column vector, from the right. For example, $Q_{4,2}$ is the determinant of the following matrix
\begin{equation}\label{cp71}
  \begin{bmatrix}
    \rho_1^s-1 &\rho_1^s &\rho_1^s &\rho^s_1 \\
    \rho_2^s &0 &0 &0\\
    0 &\rho_3^s-1 &\rho_3^s &\rho^s_3 \\
    \rho_4^s &\rho_4^s &\rho_4^s-1 &0
    \end{bmatrix}.
\end{equation}

If we calculate $Q_{2m-2,i}(s)$ by expanding the corresponding determinant by the second from the last row, it is easy to see that depending on $i\in [2m-2]$ we obtain the following values 
\begin{multline}\label{cp73}
  Q_{2m-2,i}(s)=\\
  =\begin{cases}
    (1-\rho^s_{2m-2})Q_{2m-4,i}(s),\quad i\in [2m-4] \\
    \rho^s_{2m-3}(1-\rho^s_{2m-2})Q_{2m-4}(s),\quad i=2m-3 \\
    \rho^s_{2m-2}(\sum_{j=1}^{2m-4}(-1)^j Q_{2m-4,j}(s) -\rho^s_{2m-3}Q_{2m-4}(s)),\quad i=2m-2.
  \end{cases}
\end{multline}
With the help of \eqref{cp74} and \eqref{cp73} we can construct $Q_{2m-2}(s)$ for any $m\geq 3$. That is, with this recursive algorithm we can calculate the Hausdorff dimension of any CPLIFS that satisfies the IOSC if its functions only break at their respective fixed points.

\section{Connection to expansive systems}\label{sf89}

It is easy to see that for every IFS $\mathcal{F}$ there exists a unique "smallest" non-empty compact interval $I^{\mathcal{F}}$  which is sent into itself by all the mappings of $\mathcal{F}$:

 \begin{equation}
 \label{cr22}
I^{\mathcal{F}}:=\bigcap\left\{ J:
J\subset \mathbb{R}\mbox{ compact interval with }
f_k(J)\subset J \mbox{, for all }k\in[m]
 \right\},
 \end{equation}
 where $[m]:=\left\{ 1,\dots  ,m \right\}$. 
It is easy to see that 
\begin{equation}
\label{cr15}
\Lambda^{\mathcal{F}}=\bigcap\limits_{n=1}^{\infty}
\bigcup\limits_{(i_1,\dots  ,i_n)\in[m]^n}
I_{i_1\dots  i_n}^{\mathcal{F}} ,
\end{equation}
where $I_{i_1\dots  i_n}^{\mathcal{F}}:=
f_{i_1\dots  i_n}(I^{\mathcal{F}})$ are the \texttt{cylinder intervals}.

\begin{definition}\label{sf87}
  Let $\mathcal{F}=\{ f_k\}_{k=1}^m$ be a CPLIFS with first cylinder intervals $I_1,\dots ,I_m$. We say that $\mathcal{F}$ satisfies the \texttt{Interval Open Set Condition (IOSC)} if 
  \[
    \forall i,j\in[m]: I_i \bigcap I_j =\emptyset.
  \]  
\end{definition}

In \cite{prokaj2021piecewise} we showed that typically the breaking points of a CPLIFS do not fall onto the attractor, and in this case the dimensions of the attractor coincide. 
Assuming the interval open set condition we will show that the Hausdorff dimension of the attractor of an injective CPLIFS equals to its natural dimension, independently of the position of the breaking points. Then it follows from Corollary \ref{cr12} that the box and Hausdorff dimensions of the attractor are all equal.

To handle these piecewise linear systems we use the notion of P. Raith \cite{raith1994continuity} and F. Hofbauer \cite{hofbauer1996box}, and turn our attention to the associated expanding maps. 

Let $\mathcal{F} =\lbrace f_i\rbrace_{i=1}^m$ be an injective CPLIFS that satisfies the IOSC, and without loss of generality assume $I^{\mathcal{F}}=[0,1]$.
As usual, we write $I_{i_1\dots i_k} =f_{i_1\dots i_k}(I)$ and $\Lambda$ for the attractor of $\mathcal{F}$. Recall, that we denoted with $B(k)$ the set of breaking points of $f_k$. We simply write
$f_k(B(k)):=\lbrace f_k(x):\: x\in B(k)\rbrace$ for the set of the images of the breaking points of $f_k,\; k\in\left[ m\right]$.

Let $T: \cup_{k\in[m]} I_k \rightarrow \left[ 0,1\right]$ be defined as follows
\begin{equation}
\forall k\in\left[ m\right] \: :\:
\left( T\vert_{I_k}\right)^{-1} =f_k.
\label{eq:locinverse}
\end{equation}

\begin{figure}[h]
    \centering
    \includegraphics[width=13cm]{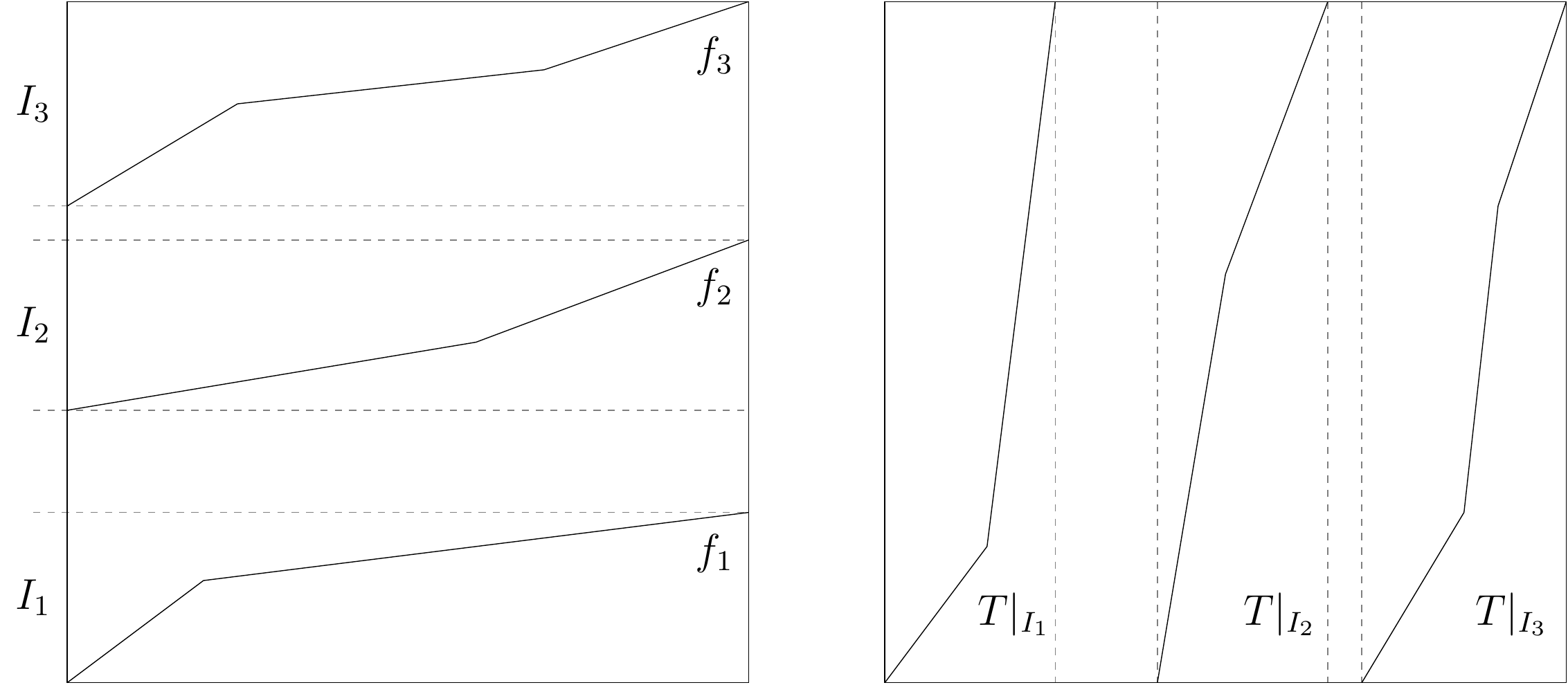}
    \caption{An injective CPLIFS that satisfies the IOSC and its associated expansive mapping.}\label{sf55}
\end{figure}

Write $W_0$ for the set of points where the derivative of $T$ is not defined. Hence $W_0=\bigcup_{k=1}^m \left( \lbrace f_k(0), f_k(1)\rbrace \cup f_k(B(k)) \right)$. We write $W$ for the set of preimages of the elements of $W_0$: 
$$W:=\left(\bigcup_{i=0}^{\infty}
T^{-i}(W_0\setminus\lbrace 0,1\rbrace )\right)\setminus \lbrace 0,1\rbrace.$$
Now let
$$ R = \bigcap_{n=0}^\infty \left( [0,1] \setminus
T^{-n} \left( [0,1] \setminus \bigcup_{k=1}^m  I_k \right) \right) .$$ 

Thus $R$ contains all the points whose orbit will never leave the union of the first cylinders as we iterate $T$. Observe that $R = \Lambda$.

Instead of $\left[ 0,1\right]$ we will work on a different metric space, obtained by doubling some points, that we denote with $\left[ 0,1\right]_W$. Namely, following \cite[p.~41]{raith1994continuity}, we double all elements of $W$, and equip this new space with the metric that induces the order topology. We call this new complete metric space \texttt{Doubled points topology}.
We write $R_W$ for the closure of $R\setminus W$ in the doubled points topology.
Let $T_W$ be the unique, continuous extension of our expanding map $T$ in this new metric space. Similarly, for a piecewise constant function $\psi$ let $\psi_W$ denote the unique continuous function for which
$\psi_W(x)=\psi(x),\; \forall x\in \left[ 0,1\right] \setminus(W\cup \lbrace 0,1\rbrace)$. We will call $\psi_W$ the completion of $\psi$.

Let $(X,d)$ be a compact metric space, $G:\: X\rightarrow X$ be a continuous mapping, and $g\in \mathcal{C}(X,\mathbb{R})$ be a continuous real valued function. The classical topological pressure is defined as
\begin{equation}\label{cp80}
p(R,G,g) = \lim_{\varepsilon\rightarrow 0}
\limsup_{n\rightarrow\infty} \frac{1}{n} \log \sup_E
\sum_{x\in E} \exp \left( \sum_{j=0}^{n-1} g(G^j x)\right),
\end{equation}
where the supremum is taken over all $(n,\varepsilon )$-separated subsets $E$ of $R$. A set $E\subset R$ is $(n,\varepsilon )$-separated, if for every $x\neq y\in E$ there exists a $j\in\lbrace 0,1,\dots ,n-1\rbrace$ with $d(T^j x,T^j y)>\varepsilon$, where $d$ is the metric on $X$ which induces the order topology.

In the doubled points topology $\psi_s(x)=-s\log \vert T^{\prime}(x) \vert$ has a continuous completion that we will denote by $\psi_{s,W}$. Thus the pressure function \eqref{cp80} is well defined for $X=R_W,\: G=T_W,\: g=\psi_{s,W}$, where the choice of $s\in\mathbb{R}$ is arbitrary. That is why we work on this new topological space.

According to \cite[Lemma 3]{raith1994continuity} the map
\begin{equation}\label{cp82}
P_{top}(s):= p(R_W,T_W,\psi_{s,W})
\end{equation}
is continuous and strictly decreasing.
Moreover, the root of this map $s_{\rm top}$ coincides with the Hausdorff dimension of $R_W$. Since $R$ and $R_W$ only differs in countably many points $\dim_H R=\dim_H R_W$.
We call the map $P_{top}(s)$ \texttt{Topological Pressure Function}. As a consequence we obtain 
\begin{lemma}\label{cp81}
  Let $\mathcal{F}$ be a CPLIFS on the line that satisfies the IOSC, and denote its attractor with $\Lambda$. Then 
  \[ \dim_{\rm H}\Lambda = s_{\rm top},
  \]
    where $s_{\rm top}$ is the unique root of the topological pressure defined in \eqref{cp82}.
\end{lemma}

\subsection{The natural pressure function}

For $s\in [0,\infty]$, we call the function 
\begin{equation}\label{cr64}
    \Phi^{\mathcal{F}}(s):=\limsup_{n\rightarrow\infty}\frac{1}{n}\log \sum_{i_1\dots i_n} |I^{\mathcal{F}}_{i_1\dots i_n}|^s
\end{equation}
the \texttt{natural pressure function} of $\mathcal{F}$. Note that this pressure can be defined for any IFS on the line.

It is easy to see that one obtain  $\Phi ^{\mathcal{F}}(s)$
above as a special case of the non-additive upper capacity topological pressure introduced by Barreira in \cite[p. 5]{barreira1996non}. 
According to \cite[Theorem 1.9]{barreira1996non}, the zero of $\Phi^{\mathcal{F}}(s)$ is well defined
\begin{equation}\label{cr61}
    s_{\mathcal{F}}:=(\Phi^{\mathcal{F}})^{-1}(0).
\end{equation}
For a given IFS $\mathcal{F}$ on the line, we name $s_{\mathcal{F}}$
the \texttt{natural dimension of the system}. Barreira also showed that $s_{\mathcal{F}}$ is always bigger or equal to the upper box dimension of the attractor $\Lambda$ of $\mathcal{F}$.
\begin{corollary}[Barreira]\label{cr12}
  For any IFS $\mathcal{F}$ on the line 
  \begin{equation}
\label{cr13}
\overline{\dim}_{\rm B}  \Lambda^{\mathcal{F}}\leq
s_{\mathcal{F}}.
\end{equation}
\end{corollary}
For a $\mathcal{F}=\lbrace f_i\rbrace_{i=1}^m$ CPLIFS let $(s_n)_{n\geq 1}$ be the unique series for which
\begin{equation} \label{cp84}
\sum_{i_1\dots i_n} \vert I_{i_1\dots i_n}\vert^{s_n} =1
\end{equation}
holds for every $n\geq 1$.
The following lemma shows that $s=\limsup_{n\rightarrow\infty } s_n$ equals to the root of the natural pressure $s_{\mathcal{F}}$. Recall that we write $\Lambda^{\mathcal{F}}$ for the attractor of the CPLIFS $\mathcal{F}$.

\begin{lemma}\label{cp85}
For a $\mathcal{F}$ CPLIFS defined on $\left[ 0,1\right]$, let $(s_n)_{n\geq 1}$ be the unique series that satisfies \ref{cp84}. Then the following holds
$$ s_{\mathcal{F}} = \limsup_{n\rightarrow\infty} s_n ,$$
where $s_{\mathcal{F}}$ is the root of the natural pressure function $\Phi^{\mathcal{F}}$.
\end{lemma}

\begin{proof}
Recall that we denote the smallest and largest contraction ratio of $\mathcal{F}$ by $\rho_{\min}$ and $\rho_{\max}$ respectively, and fix an arbitrary $n\geq 1$.

For a given $n$ length word $i_1\dots i_n$ and an arbitrary $s$ we can write
$\vert I_{i_1\dots i_n}\vert^s = \vert I_{i_1\dots i_n}\vert^{s_n}
\cdot \vert I_{i_1\dots i_n}\vert^{s-s_n}$ to obtain the estimates
\begin{equation}\label{cp79}
\rho_{\min}^{n(s-s_n)}\cdot \vert I_{i_1\dots i_n}\vert^{s_n}
\leq \vert I_{i_1\dots i_n}\vert^s \leq
\rho_{\max}^{n(s-s_n)}\cdot \vert I_{i_1\dots i_n}\vert^{s_n}
\end{equation}

Since by definition $\sum_{i_1\dots i_n}\vert I_{i_1\dots i_n}\vert^{s_n} =1$, equation \eqref{cp79} implies

\begin{equation}
(s-s_n)\log \rho_{\min} \leq
\frac{1}{n}\log \sum_{i_1\dots i_n} \vert I_{i_1\dots i_n}\vert^s \leq
(s-s_n)\log \rho_{\max}
\label{eq:snat_est2}
\end{equation}

Reordering the inequalities we obtain
\begin{equation}
\frac{\frac{1}{n}\log \sum_{i_1\dots i_n} \vert I_{i_1\dots i_n}\vert^s}{\log \rho_{\max}} \leq
(s-s_n) \leq
\frac{\frac{1}{n}\log \sum_{i_1\dots i_n} \vert I_{i_1\dots i_n}\vert^s}{\log \rho_{\min}}
\end{equation}

If we choose $s$ to be equal to $s_{\mathcal{F}}$, then taking the limit superior of each side yields $\limsup_{n\rightarrow\infty} s_{\mathcal{F}}-s_n =0$.

\end{proof}

Using this lemma we show that for an injective CPLIFS $\mathcal{F}$ the root of the natural pressure \eqref{cr64} coincides with the root of the topological pressure function \eqref{cp82} of the associated expanding map.

\begin{lemma}\label{cp86}
Let $\mathcal{F}$ be an injective CPLIFS that satisfies the IOSC. Write $s_{top}$ for the root of the topological pressure function \eqref{cp82}, and $s_{\mathcal{F}}$ for the root of the natural pressure function \eqref{cr64}. Then
\[
s_{top} = s_{\mathcal{F}}
\].
\end{lemma}

\begin{proof}
Recall that we denote the smallest and largest contraction ratio of $\mathcal{F}$ by $\rho_{\min}$ and $\rho_{\max}$ respectively.
Fix an $\varepsilon >0$.

By Lemma \ref{cp85} $s_{\mathcal{F}}=\limsup_{n\rightarrow\infty} s_n$,
hence $\exists N \text{ such that } \forall n>N :\: s_n < s_{\mathcal{F}}+\frac{\varepsilon}{2}$. Thus by \eqref{cp84} we have

\begin{equation}
\sum_{i_1\dots i_n} \vert I_{i_1\dots i_n}\vert^{s_{\mathcal{F}}+\varepsilon}
<  \sum_{i_1\dots i_n} \vert I_{i_1\dots i_n}\vert^{s_n+\frac{\varepsilon}{2}} \leq \rho_{\max}^{\frac{n\varepsilon}{2}}
\rightarrow 0, \text{ as  } n\rightarrow\infty .
\end{equation}

It implies that $\mathcal{H}^{s_\mathcal{F}+\varepsilon}(\Lambda_{\mathcal{F}}) =0$ for each $\varepsilon >0$, where $\mathcal{H}^s$ stands for the $s$-dimensional Hausdorff measure. By the definition of the Hausdorff dimension, we obtain
\begin{equation}
\dim_H \Lambda^{\mathcal{F}} \leq s_{\mathcal{F}}
\label{eq:oneside_upper}
\end{equation}

According to Lemma \ref{cp81} $\dim_H \Lambda^{\mathcal{F}} = s_{top}$, so we already proved
$$s_{top} \leq s_{\mathcal{F}}.$$

To prove the other direction, we first need to reformulize the pressure function $P_{top}(s)$ defined in \eqref{cp82}, to see how it relates to the natural pressure $\Phi^{\mathcal{F}}(s)$ \eqref{cr64}.

Recall that we assumed that the IOSC holds, which means that all of the level $1$ cylinders are separated by some positive distance $D>0$.
Fix an $D >\varepsilon >0$, and let $N(\varepsilon )$ be sufficiently big such that
$$ \max_{i_1,\dots ,i_{N(\varepsilon )}} \vert I_{i_1,\dots ,i_{N(\varepsilon )}}\vert <\varepsilon .$$
We will show that by choosing one element of each level $N(\varepsilon )$ cylinder we obtain an $(N(\varepsilon ),\varepsilon)$-separated set.

For any two $N$ length words $\mathbf{i}=\lbrace i_1,\dots i_N\rbrace ,\; \mathbf{j}=\lbrace j_1,\dots j_N\rbrace$ let
$ \vert \mathbf{i} \wedge \mathbf{j}\vert = \min \lbrace k-1:\:
i_k \neq j_k\rbrace$ .
Thus if we iterate $\vert \mathbf{i} \wedge \mathbf{j}\vert$ times $T$ over the cylinders
$I_{i_1,\dots ,i_N} ,\: I_{j_1,\dots ,j_N}$, the images will fall into different first level cylinders. More formally
$$
d(T^{\vert \mathbf{i} \wedge \mathbf{j}\vert} I_{\mathbf{i}} ,\:
T^{\vert \mathbf{i} \wedge \mathbf{j}\vert} I_{\mathbf{j}}) > D >\varepsilon .
$$

Therefore by choosing one element from each $N(\varepsilon )$ level cylinder, we obtain an $(N(\varepsilon ),\varepsilon)$-separated subset that we denote by $I^{N(\varepsilon )}_{sep}$.
We require $\forall x\in I^{N(\varepsilon )}_{sep}$ to maximize the derivative of $T$ over the $N(\varepsilon )$ cylinder that contains $x$. We can make this constraint, since any choice of elements will do.
Remember that we use the doubled points topology introduced in \cite[p.~41]{raith1994continuity}, so $T^{\prime}$ is well defined at every $x\in \left[ 0,1\right]_W$ .

We can define $I^{n}_{sep}$ similarly for any $n>N(\varepsilon)$.
We substitute these sets into the topological pressure to gain a lower bound. We used the notation $\psi_s(x) := -s\log\vert T^{\prime}(x)\vert$ to make the formulas more concise.

\begin{align*}
P_{top}(s) &= \lim_{\varepsilon\rightarrow 0}\limsup_{n\rightarrow\infty}
\frac{1}{n}\log
\sup_E \sum_{x\in E} \exp\left( \sum_{j=0}^{n-1} \psi_s(T^j x)\right) \\
&\geq \lim_{\varepsilon\rightarrow 0}\limsup_{n \rightarrow\infty}
\frac{1}{n}\log \sum_{x\in I^{n}_{sep}}
\exp\left( \sum_{j=0}^{n-1} \psi_s(T^j x)\right) \\
&\geq \lim_{\varepsilon\rightarrow 0}\limsup_{n \rightarrow\infty}
\frac{1}{n}\log \sum_{\mathbf{i}:\:
\Pi (\mathbf{i})\in I^{n}_{sep}}
\exp\left( \sum_{j=0}^{n-1} \vert I_{\mathbf{i}}\vert^s \right)
= \Phi^{\mathcal{F}} (s),
\label{eq:oneside_lower}
\end{align*}
where in the last inequality we substituted $\psi_s(x)=-s\log \vert T^{\prime} (x)\vert $.
For $s=s_{\mathcal{F}}$, the right hand side is equal to $0$. The pressure function $P_{top}(s)$ is strictly decreasing, thus its unique zero $s_{top}$ must be bigger or equal to $s_{\mathcal{F}}$.
We just obtained $$ s_{\mathcal{F}} \leq s_{top} .$$

\end{proof}

As a consequence of Lemma \ref{cp86}, Lemma \ref{cp81} and Corollary \ref{cr12},
we obtain
\begin{theorem}\label{cs13}
If $\Lambda$ is the attractor of a CPLIFS $\mathcal{F}$ that satisfies the IOSC, then
\begin{equation}\label{sf76}
  \dim_H \Lambda = \dim_B \Lambda =s_{\mathcal{F}},
\end{equation}
where $s_{\mathcal{F}}$ is the unique root of the natural pressure function $\Phi^{\mathcal{F}}$ defined in \eqref{cr61}.
\end{theorem}

Note that the equality of the box and Hausdorff dimensions of the attractor also follows from Lemma \eqref{cp81} and the Main theorem of \cite{hofbauer1996box}.

\section{Lebesgue measure of the attractor for small parameters}

Let $\mathcal{F}=\{f_k\}_{k=1}^m$ be a CPLIFS. Recall that $\rho_k$ is the largest expansion ratio of $f_k$ in absolute value, and $\rho_{\max}=\max_{k\in[m]} \rho_k$. 
Throughout this section we will assume that all contraction ratios of the functions of $\mathcal{F}$ are positive.  

\begin{definition}\label{sf67}
  We call $\mathcal{F}$ \texttt{small} if both of the following two requirements hold:
  \begin{enumerate}[{\bf (a)}]
    \item $\sum\limits_{k=1}^{m}\rho_k<1$.
    \item  Our second requirement depends on the injectivity of $f_k$:
      \begin{enumerate}[{\bf (i)}]
        \item If $f_k$ is injective then we require that $\rho_k<\frac{1}{2}$.
        \item If $f_k$ is not injective then we require that $\rho_k<\frac{1-\rho_{\max}}{2}$, which always holds if $\rho_{\max}<\frac{1}{3}$.
      \end{enumerate}
    \end{enumerate}
\end{definition}

According to Proposition 2.3 of \cite{prokaj2021piecewise}, we may represent a $\dim_{\rm P}$-typical small CPLIFS with a self-similar GDIFS. Using some lemmas from \cite{prokaj2021piecewise} and Theorem \ref{sf72}, we are going to show that $s_{\mathcal{F}}>1$ typically implies that the Lebesgue measure of the attractor $\Lambda_{\mathcal{F}}$ is positive.

\begin{theorem}\label{sf63}
  Let $(\mathfrak{b},\pmb{\tau})$ be the vector of translation parameters of a system $\mathcal{F}=\{f_k\}_{k=1}^m\in \mathrm{CPLIFS}_{\pmb{\ell} }$ with attractor $\Lambda^{(\mathfrak{b},\pmb{\tau})}$. If all the functions in $\mathcal{F}$ has positive slopes, then for $\mathcal{L}_{L+m}$-almost every $(\mathfrak{b},\pmb{\tau})\in \mathfrak{B}^{\pmb{\ell} }\times \mathbb{R}^m$ we have 

  \begin{equation}\label{sf66}
    s_{\mathcal{F}}>1 \implies \mathcal{L}_1(\Lambda^{(\mathfrak{b},\pmb{\tau})}) >0,
  \end{equation}
  where $L=l(1)+\dots+l(m)$ is the total number of breaking points in $\mathcal{F}$.
\end{theorem}

\begin{proof}
  Fix a small vector of contractions $\pmb{\rho}$. Let $\mathcal{T}^{\pmb{\rho}}\subset \mathfrak{B}^{\pmb{\ell} }\times \mathbb{R}^m$ be the set of those translation parameters $(\mathfrak{b},\pmb{\tau})$ for which the associated CPLIFS $\mathcal{F}^{(\mathfrak{b},\pmb{\tau})}$ is regular. By \cite[Proposition~2.3]{prokaj2021piecewise}, $\mathcal{T}^{\pmb{\rho}}$ has total measure with respect to $\mathcal{L}_{L+m}$. 

  We say that two countinuous piecewise linear iterated function systems $\mathcal{F}$ and $\mathcal{F}^{'}$ are equivalent if they are defined by the same directed graph, and they have the same contractions on every edge. Equivalent CPLIFSs are not necessarily identical, as they might have different graph directed sets. 
  For an arbitrary $(\mathfrak{b},\pmb{\tau})\in\mathcal{T}^{\pmb{\rho}}$, we define $\mathcal{T}^{\pmb{\rho}}_{(\mathfrak{b},\pmb{\tau})}\subset \mathcal{T}^{\pmb{\rho}}$ as the equivalence class of $(\mathfrak{b},\pmb{\tau})$. These neighbourhoods form an open cover of $\mathcal{T}^{\pmb{\rho}}$.

  Now we are left to prove that for any $(\mathfrak{b},\pmb{\tau})\in\mathcal{T}^{\pmb{\rho}}$, \eqref{sf66} holds for $\mathcal{L}_{L+m}$-almost every $(\mathfrak{b}^{'},\pmb{\tau}^{'})\in \mathcal{T}^{\pmb{\rho}}_{(\mathfrak{b},\pmb{\tau})}$. 
  
  Let us pick an arbitrary $(\mathfrak{b},\pmb{\tau})$. Then, $\mathcal{F}:=\mathcal{F}^{(\mathfrak{b},\pmb{\tau})}$ has an associated graph-directed system which we denote by $\mathcal{H}=\{h_e=\lambda_e+t_e\}_{e\in\mathcal{E}}$, where $\mathcal{G}=(\mathcal{V},\mathcal{E})$ is the directed graph that defines $\mathcal{H}$. 

  According to \cite[Fact~4.1]{prokaj2021piecewise}, the function 
  \begin{equation}\label{sf65}
    \Psi_{\pmb{\rho}}(\mathfrak{b},\pmb{\tau}) = \mathbf{t}
  \end{equation}
  that assigns the translation vector $\mathbf{t}$ of $\mathcal{H}$ to the translation parameters $(\mathfrak{b},\pmb{\tau})$ of $\mathcal{F}$ is a non-singular affine transformation. 
  As $\alpha(\mathcal{H})=s_{\mathcal{F}}$ by \cite[Lemma~5.1]{prokaj2021piecewise}, for $\mathcal{L}_{\vert \mathcal{E}\vert}$-almost every $\mathbf{t}\in \Psi_{\pmb{\rho}}(\mathcal{T}^{\pmb{\rho}}_{(\mathfrak{b},\pmb{\tau})})$ \eqref{sf66} holds by Theorem \ref{sf72}. 
  That is, \eqref{sf66} holds for $\mathcal{L}_{L+m}$-almost every $(\mathfrak{b}^{'},\pmb{\tau}^{'})\in \mathcal{T}^{\pmb{\rho}}_{(\mathfrak{b},\pmb{\tau})}$.
\end{proof}

\bibliographystyle{abbrv}
\bibliography{bibl_tort_vonal}


\end{document}